\documentclass[14pt]{amsart}

\usepackage{amsmath,amssymb,amsfonts,enumerate,amsthm, amscd}
\usepackage[french, ngerman, italian, english]{babel}
\usepackage{hyperref}
\usepackage{paralist}
\usepackage{verbatim}
\usepackage{tikz}
\usetikzlibrary{arrows,decorations.pathmorphing,backgrounds,positioning,fit}
\RequirePackage{pgfrcs}
\theoremstyle{plain}
\newtheorem{thm}{\bf Theorem}[section]

\newtheorem{prop}[thm]{\bf Proposition}
\newtheorem{lem}[thm]{\bf Lemma}
\newtheorem{cor}[thm]{\bf Corollary}

\theoremstyle{definition}
\newtheorem{defn}[thm]{\bf Definition}
\theoremstyle{remark}
\newtheorem{rem}[thm]{\bf Remark}
\newtheorem{exam}[thm]{\bf Example}

\theoremstyle{example}

\def \supp{{\mathrm{supp}}}
\def \Supp{{\mathrm{Supp}}}
\def \s{{\mathrm{star}}}
\def \lk{\mathrm{lk}}

\def \n{\mathfrak n}

\def \a{\mathbf a}
\def \b{\mathbf b}
\def \u{\mathbf u}
\def \v{\mathbf v}
\def \M{\mathrm{M}}
\def \1{\mathbf 1}
\def \c{\mathbf{c}}

\def \Ass{\mathrm{Ass}}
\def \ass{\mathrm{ass}}
\def \Min{\mathrm{Min}}

\def \fpt{\mathrm{fpt}}
\def \infpt{\mathrm{infpt}}
\def \lk{\mathrm{link}}
\def \Del{\Delta}

\def \NN{\mathbb N}
\def \T{\mathcal T}
\def \ZZ{\mathbb Z}

\def \F{\mathcal F}

\def \M{\mathcal M}

\begin{document}

\title[]{Pretty $k$-clean monomial ideals and $k$-decomposable multicomplexes}
\author{Rahim Rahmati-Asghar}


%\thanks{The research of S. Yassemi was in part supported by a grant from IPM No. 91130214}

\date{\today}

\begin{abstract}
We introduce pretty $k$-clean monomial ideals and $k$-decomposable multicomplexes, respectively, as the extensions of the notions of $k$-clean monomial ideals and $k$-decomposable simplicial complexes. We show that a multicomplex $\Gamma$ is $k$-decomposable if and only if its associated monomial ideal $I(\Gamma)$ is pretty $k$-clean. Also, we prove that an arbitrary monomial ideal $I$ is pretty $k$-clean if and only if its polarization $I^p$ is $k$-clean. Our results extend and generalize some results due to Herzog-Popescu, Soleyman Jahan and the current author.
\end{abstract}

\maketitle

\section*{Introduction}

Let $R$ be a Noetherian ring and $M$ be a finitely generated $R$-module. It is
well-known that there exists a so called \emph{prime filtration}
$$\F:0=M_0\subset M_1\subset \ldots\subset M_{r-1}\subset M_r=M$$
that is such that $M_i/M_{i-1}\cong R/P_i$ for some $P_i\in\Supp(M)$. The set $\{P_1,\ldots,P_r\}$ is called the \emph{support} of $M$ and denoted by $\Supp(\F)$. Let $\Min(M)$ denote the set of minimal prime ideals in $\Supp(M)$. Dress \cite{Dr} calls a prime filtration $\F$ of $M$ \emph{clean} if $\Supp(\F)=\Min(M)$. The module $M$ is called clean, if $M$ admits a clean filtration and $R$ is clean if it is a clean module over itself.

Herzog and Popescu \cite{HePo} introduced the concept of \emph{pretty clean} modules. A prime filtration
$$\F:0=M_0\subset M_1\subset \ldots\subset M_{r-1}\subset M_r=M$$
of $M$ with $M_i/M_{i-1}\cong R/P_i$ is called \emph{pretty clean}, if for all $i<j$ for which $P_i\subseteq P_j$ it follows
that $P_i=P_j$. The module $M$ is called pretty clean, if it has a pretty clean filtration. We say an ideal $I\subset R$ is clean (or pretty clean) if
$R/I$ is clean (or pretty clean).

Dress showed \cite{Dr} that a simplicial complex is shellable if and only if its Stanley-Reisner
ideal is clean. This result was extended in two different forms by Herzog and Popescu in \cite{HePo} and also by the current author in \cite{Ra}. Herzog and Popescu showed that a multicomplex is shellable if and only if its associated monomial ideal is pretty clean (see \cite[Theorem 10.5.]{HePo}) and we proved that a simplicial complex is $k$-decomposable if and only if its Stanley-Reisner ideal is $k$-clean (see \cite[Theorem 4.1.]{Ra}). Pretty cleanness and $k$-cleanness are, respectively, the algebraic counterpart of shellability for multicomplexes due to \cite{HePo} and $k$-decomposability for simplicial complexes due to Billera-Provan \cite{BiPr} and Woodroofe \cite{Wo}.
%The concept of $k$-clean monomial ideals was introduced in \cite{Ra} as an extension of clean monomial ideals and is in fact the algebraic counterpart of $k$-decomposability for simplicial complexes due to
%Billera and Provan defined the notion of $k$-decomposability for simplicial complexes \cite{BiPr}. Then Woodroofe extended this notion for non-pure simplicial complexes \cite{Wo}. In \cite{Ra} we introduced the concept of $k$-clean monomial ideals as an extension of clean monomial ideals and proven that a simplicial complex is $k$-decomposable if and only if its Stanley-Reisner ideal is $k$-clean.
Soleyman Jahan proved that a monomial ideal is pretty clean if and only if its polarization is clean (see \cite[Theorem 3.10.]{So}). This yields a characterization of pretty clean monomial ideals, and it also
implies that a multicomplex is shellable if and only the simplicial complex corresponding to its polarization
is (non-pure) shellable. The purpose of this paper is to improve and generalize these results. To this end we introduce two notions: pretty $k$-clean monomial ideal and $k$-decomposable multicomplex. The first notion is as an extension of pretty clean monomial ideals as well as $k$-clean monomial ideals and the second one extends two notions shellable multicomplexes and $k$-decomposable simplicial complexes. The new constructions introduced here imply that pretty clean monomial ideals and shellable multicomplexes have recursive structures and, moreover, determine more details of their combinatorial properties.

The paper is organized as follows. In the first section, we review some preliminaries which are needed in the sequel. In the second section, we define pretty cleaner monomials, which naturally leads us to define pretty $k$-clean monomial ideals.  We show that

\textbf{Theorem \ref{Equ. pretty k-clean}.} A pretty $k$-clean monomial ideal is pretty clean and also every pretty clean monomial ideal is pretty $k$-clean for some $k\geq 0$.

This theorem implies that pretty $k$-cleanness is an extension of pretty cleanness and, moreover, since pretty $k$-clean monomial ideals have a recursive structure it follows that pretty clean ideals have such a property, too.

In the third section we define a class of multicomplexes, called $k$-decomposable multicomplexes and discuss some structural properties of them. We prove that

\textbf{Theorem \ref{shell k-decom}.} Every $k$-decomposable multicomplex is shellable and every shellable multicomplex is $k$-decomposable for some $k\geq 0$.

In Proposition \ref{Sim-multi k-decom} we show that our definition of $k$-decomposable multicomplexes extends the corresponding notion known for simplicial complexes due to  Billera and Provan \cite{BiPr} or Woodroofe \cite{Wo}.

The final section is devoted to the main results of the paper. As the first main result, we show that

\textbf{Theorem \ref{k-decom pretty}.} A multicomplex $\Gamma$ is $k$-decomposable if and only if its associated monomial ideal $I(\Gamma)$ is pretty $k$-clean.

This result generalizes Theorem 10.5 of \cite{HePo} and also Theorem 4.1 of \cite{Ra} and, moreover, it implies that Theorem \ref{shell k-decom} is a combinatorial translation of Theorem \ref{Equ. pretty k-clean}. To obtain the second main result of section 4, we first prove that a multicomplex is $k$-decomposable if and only if its polarization is $k$-decomposable (see Theorem \ref{polar k-decom}). This leads us to prove that

\textbf{Corollary \ref{polar of k-clean}.} A monomial ideal $I$ is pretty $k$-clean if and only if its polarization $I^p$ is $k$-clean.

This extends Theorem 3.10 of \cite{So} which says that an arbitrary monomial ideal $I$ is pretty clean if and only if its polarization is clean.

Our proofs here are often combinatorial and in this way we introduce the new features of the structure of pretty clean monomial ideals.

\section{preliminaries}

Let $S=K[x_1,\ldots,x_n]$ be the polynomial ring over a field $K$. Let $I\subset S$ be a monomial ideal. Set $\ass(I)=\Ass(S/I)$ and $\min(I)=\Min(S/I)$. A prime filtration of $I$ is of the form
$$\F:I=I_0\subset I_1\subset\ldots\subset I_r=S$$
with $I_j/I_{j-1}\cong S/P_j$, for $j=1,\ldots,r$ such that all $I_j$ are monomial ideals.

The prime filtration $\F$ is called \emph{clean}, if $\Supp(\F)=\min(I)$. Also, $\F$ is called \emph{pretty clean}, if for all $i<j$ which $P_i\subseteq P_j$ it
follows that $P_i=P_j$. The monomial ideal $I$ is called clean(or pretty clean), if it has a clean (or pretty clean) filtration. It was shown in \cite{HePo} that if $\F$ is a pretty clean filtration of $I$ then $\Supp(\F)=\ass(I)$.

Let $\Del$ be a simplicial complex on the vertex set $[n]:=\{x_1, \dots, x_n\}$. The set of facets (maximal faces) of $\Del$
is denoted by $\mathcal{F}(\Del)$ and if $\mathcal{F}(\Del)=\{F_1,\ldots,F_r\}$, we write $\Del=\langle F_1,\ldots,F_r\rangle$. For a monomial ideal $I$ of $S$, the set of minimal generators of $I$ is denoted by $G(I)$.

\begin{defn}
A simplicial complex $\Del$ is called \emph{shellable} if there exists an ordering $F_1,\ldots,F_m$ on the
facets of $\Delta$ such that for any $i<j$, there exists a vertex
$v\in F_j\setminus F_i$ and $\ell<j$ with
$F_j\setminus F_\ell=\{v\}$. We call $F_1,\ldots,F_m$ a \emph{shelling} for $\Del$.
\end{defn}

\begin{thm}
\cite{Dr} The simplicial complex $\Del$ is shellable if and only if its Stanley-Reisner ideal $I_\Del$ is
a clean monomial ideal.
\end{thm}

For a simplicial complex $\Del$ and $F\in \Del$, the link of $F$ in
$\Del$ is defined as $$\lk_{\Delta}(F)=\{G\in \Del: G\cap
F=\emptyset, G\cup F\in \Del\},$$ and the deletion of $F$ is the
simplicial complex $$\Del\setminus F=\{G\in \Delta: F \nsubseteq G\}.$$

Woodroofe in \cite{Wo} extended the definition of $k$-decomposability
to non-pure complexes as follows.

Let $\Delta$ be a simplicial complex on vertex set $X$. Then a face $\sigma$ is called a
\emph{shedding face} if every face $\tau$  containing $\sigma$ satisfies the following exchange property: for every
$v \in \sigma$ there is $w\in X \setminus \tau$ such that $(\tau \cup \{w\})\setminus \{v\}$ is a face of $\Delta$.

\begin{defn}\label{k-decom simcomplex}
\cite{Wo} A simplicial complex $\Delta$ is recursively defined to be \emph{$k$-decomposable} if
either $\Delta$ is a simplex or else has a shedding face $\sigma$ with $\dim(\sigma)\leq k$ such that both $\Delta \setminus \sigma$
and $\lk_{\Delta}(\sigma)$ are $k$-decomposable. The complexes $\{\}$ and $\{\emptyset\}$  are considered to be
$k$-decomposable for all $k \geq -1$.
\end{defn}

%\begin{rem}
%By Proposition 3.8 of \cite{Wo}, for $\sigma\in\Del$, $\lk_\Del\sigma$ is $k$-decomposable if and only if $\s_\Del\sigma$ is. It follows that in Definition \ref{k-decom simcomplex}, we can replace the notion ``$\lk_\Del\sigma$'' by ``$\s_\Del\sigma$''. We use this fact in extending Definition \ref{k-decom simcomplex} to multicomplexes.
%\end{rem}

\begin{defn}
\cite{Ra} Let $I\subset S$ be a monomial ideal. A non unit monomial $u\not\in I$ is called a \emph{cleaner monomial} of $I$ if $\min(\ass(I+Su))\subseteq\min(\ass(I))$.
\end{defn}

\begin{defn}
\cite{Ra} Let $I\subset S$ be a monomial ideal. We say that $I$ is \emph{$k$-clean} whenever $I$ is a prime ideal or $I$ has no embedded prime ideals and there exists a cleaner monomial $u\not\in I$ with $|\supp(u)|\leq k+1$ such that both $I:u$ and $I+Su$ are $k$-clean.
\end{defn}

\begin{thm}\label{k-dec k-cl}
\cite[Theorem 4.1.]{Ra} Let $\Del$ be a ($d-1$)-dimensional simplicial complex. Then $\Del$ is $k$-decomposable if and only if $I_\Del$ is $k$-clean, where $0\leq k\leq d-1$.
\end{thm}

The concept of multicomplex was first defined by Stanley \cite{St}. Then Herzog and Popescu \cite{HePo} gave a modification of Stanley's definition which will be used in this paper.

Let $\NN$ be the set of non-negative integers. Define on $\NN^n$ the partial order given by
\begin{center}
$\a\preceq\b$ if $\a(i)\leq\b(i)$ for all $i$.
\end{center}
Set $\NN_\infty=\NN\cup\{\infty\}$. For $\a\in\NN^n_\infty$ we define $\fpt(\a)=\{i:\a(i)\neq\infty\}$ and $\infpt(\a)=\{i:\a(i)=\infty\}$ and $\fpt^\ast(\a)=\{i:0<\a(i)<\infty\}$.

Let $\Gamma$ be a subset of $\NN^n_\infty$. An element $m\in\Gamma$¡ is called \emph{maximal} if there is no $\a\in\Gamma$ with $\a\succ m$. We denote by $\M(\Gamma)$ the set of maximal elements of $\Gamma$. It was shown in \cite[Lemma 9.1]{HePo} that $\M(\Gamma)$ is finite.

\begin{defn}
A subset $\Gamma\subset\NN^n_\infty$ is called a \emph{multicomplex} if
\begin{enumerate}[\upshape (1)]
  \item for all $\a\in\Gamma$ and all $\b\in\NN^n_\infty$ with $\b\preceq\a$ it follows that $\b\in\Gamma$;
  \item for all $\a\in\Gamma$¡ there exists an element $m\in\M(\Gamma)$ such that $\a\preceq m$.
\end{enumerate}
\end{defn}

The elements of a multicomplex are called \emph{faces}. An element $\a\in\Gamma$ is called a \emph{facet} of $\Gamma$ if for all $m\in\M(\Gamma)$ with $\a\preceq m$ one has $\infpt(\a)=\infpt(m)$. Let $\F(\Gamma)$ denote the set of facets of $\Gamma$. The facets in $\M(\Gamma)$ are called \emph{maximal facets}.

It is clear that the set of facets and also the set of maximal facets of a multicomplex $\Gamma$ determine $\Gamma$. The \emph{monomial ideal associated to} $\Gamma$ is the ideal $I(\Gamma)$ generated by all monomials $x^\a$ such that $\a\not\in\Gamma$. Also, if $I\subset S$ is any monomial ideal then the \emph{multicomplex associated to} $I$ is defined to be $\Gamma(I)=\{\a\in\NN^n_\infty:x^\a\not\in I\}$. Note that $I(\Gamma(I))=I$ and, moreover, $\Gamma(I)$ is unique with this property. For $A=\{\a_1,\ldots,\a_r\}\subset\NN^n_\infty$, we denote by $\langle \a_1,\ldots,\a_r\rangle$ the unique smallest multicomplex containing $A$.

For $\a\in\Gamma$, define $\dim(\a)=|\infpt(\a)|-1$ and
$$\dim(\Gamma)=\max\{\dim(\a):\a\in\Gamma\}.$$

We call $S\subset\NN^n_\infty$ a \emph{Stanley set of degree} $\a$ if there exist $\a\in\NN^n$ and $m\in\NN^n_\infty$ with
$m(i )\in\{0,\infty\}$ such that $S=\a+S^\ast$, where $S^\ast=\langle m\rangle$. The dimension of $S$ is defined to be $\dim(\langle m\rangle)$.

\begin{defn}
\cite{HePo} A multicomplex $\Gamma$ is \emph{shellable} if the facets of $\Gamma$ can be ordered $\a_1,\ldots,\a_r$ such that
\begin{enumerate}[\upshape (1)]
  \item $S_i=\langle\a_i\rangle\backslash\langle\a_1,\ldots,\a_{i-1}\rangle$ is a Stanley set for $i=1,\ldots,r$;
  \item If $S^\ast_i\subset S^\ast_j$ then $S^\ast_i=S^\ast_j$ or $i>j$.
\end{enumerate}
An ordering of the facets satisfying (1) and (2) is called a \emph{shelling} of $\Gamma$.
\end{defn}

\begin{thm}\label{Sim-multi shellable}
\cite[Proposition 10.3.]{HePo} Let $\Del$ be a simplicial complex with facets $F_1,\ldots,F_r$, and $\Gamma$ be the multicomplex with facets $\a_{F_1},\ldots,\a_{F_r}$. Then $\Del$ is shellable if and only if $\Gamma$ is shellable.
\end{thm}

\begin{thm}\label{HePo shell}
\cite[Theorem 10.5]{HePo} The multicomplex $\Gamma$ is shellable if and only if $I(\Gamma)$ is a pretty clean monomial ideal.
\end{thm}

Let $I\subseteq S$ be a monomial ideal generated by the set
$G(I)=\{u_1,\ldots,u_r\}$. Let for each $i$,
$u_i=\prod_{j=1}^{n}x^{t_{ij}}_j$ and for each $j$,
$t_j=\max\{t_{ij}:i=1,\ldots,r\}$. Let
$$T=K[x_{1,1},x_{1,2},\ldots,x_{1,t_1},x_{2,1},x_{2,2},\ldots,x_{2,t_2},\ldots,x_{n,1},x_{n,2},\ldots,x_{n,t_n}]$$
be a polynomial ring over $K$. For each $i=1,\ldots,r$ set
$$v_i:=\prod_{j=1}^{n}\prod_{k=1}^{t_{ij}}x_{jk}.$$
The monomial $v_i$ is squarefree and is called the
\emph{polarization} of $u_i$. Also, we denote the polarization of $I$
by $I^p$ and it is a squarefree monomial ideal generated by
$\{v_1,\ldots,v_r\}$.

\begin{thm}\label{polar of clean}
\cite[Theorem 3.10.]{So} The monomial ideal $I$ is pretty clean if and only if $I^p$ is clean.
\end{thm}

\section{Pretty $k$-clean monomial ideals}

Let $I\subset S$ be a monomial ideal. A prime filtration
$$\F:(0)=M_0\subset M_1\subset \ldots\subset M_{r-1}\subset M_r=S/I$$
of $S/I$ is called \emph{multigraded}, if all $M_i$ are multigraded submodules of $M$, and if there are multigraded isomorphisms $M_i/M_{i-1}\cong S/P_i(-\a_i)$ with some $\a_i\in\ZZ^n$ and some multigraded prime ideals $P_i$.

\begin{defn}\label{Equ. pretty k-clean}
Let $I\subset S$ be a monomial ideal. A non unit monomial $u\not\in I$ is called \emph{pretty cleaner} if for $P\in \ass(I:u)$ and $Q\in \ass(I+Su)$ which $P\subseteq Q$ it follows that $P=Q$.
\end{defn}

\begin{defn}\label{Equ. pretty k-clean}
A monomial ideal $I\subset S$ is called \emph{pretty $k$-clean} if it is a prime ideal or there exists a pretty cleaner monomial $u\not\in I$ with $|\supp(u)|\leq k+1$ such that both $I:u$ and $I+Su$ are pretty $k$-clean.
\end{defn}

Note that pretty $k$-cleanness implies pretty $k'$-cleanness for $0\leq k\leq k'$. But the converse implication is not true in general. To see an example of pretty $k$-clean ideals which are not pretty $0$-clean, refer to Remark \ref{prt not k-prt}.

\begin{rem}
It is clear that every $k$-clean monomial ideal is pretty $k$-clean. But a cleaner monomial need not be pretty cleaner. To see this, consider the monomial ideal $$I=(x_1x^2_2,x_2x^2_3, x^2_1x_3)\subset S'=K[x_1,x_2,x_3].$$
Then
$$\begin{array}{l}
    \ass(I)=\{(x_1,x_2),(x_1,x_3),(x_2,x_3),(x_1,x_2,x_3)\}, \\
    \ass(I+S'x^2_1)=\{(x_1,x_2),(x_1,x_3),(x_1,x_2,x_3)\},\\
    \ass(I:x^2_1)=\{(x_2,x_3)\}.
  \end{array}
$$
Notice that $x^2_1$ is cleaner but not pretty cleaner.
\end{rem}

It follows from the definition that the construction of a pretty $k$-clean monomial ideal is similar to that of a $k$-clean monomial ideal (c.f. \cite{Ra}). In other words, for a pretty $k$-clean monomial ideal $I\subset S$ there is a rooted, finite, directed and binary tree $\mathcal{T}$ whose root is $I$ and every node $\n$ is labeled by a pretty $k$-clean monomial ideal $I_\n$ containing $I$. Also, every nonterminal node $\n$ is labeled by a monomial $u_\n$ which is a pretty cleaner monomial of $I_\n$. $\T$ is depicted in the following:
$$
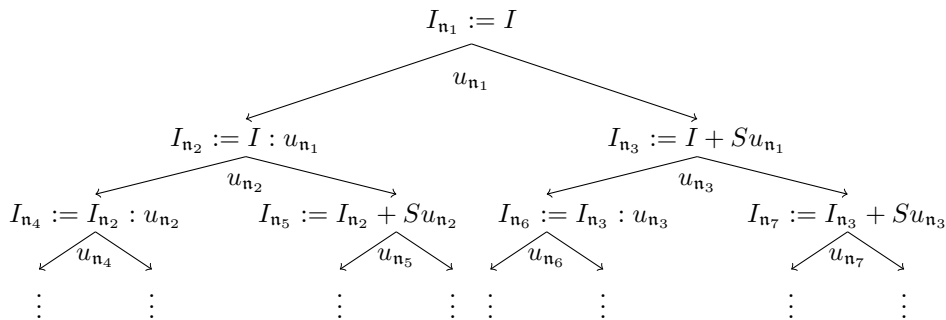
\begin{tikzpicture}
\coordinate (a) at (0,4);
\coordinate (b) at (-3,3);
\coordinate (b') at (-3,2.5);
\coordinate (c) at (3,3);
\coordinate (c') at (3,2.5);
\coordinate (d) at (-5,2);
\coordinate (e) at (-1,2);
\coordinate (f) at (1,2);
\coordinate (g) at (5,2);
\coordinate (d') at (-5,1.5);
\coordinate (e') at (-1,1.5);
\coordinate (f') at (1,1.5);
\coordinate (g') at (5,1.5);
\coordinate (h) at (-5.75,1);
\coordinate (i) at (-4.25,1);
\coordinate (j) at (-1.75,1);
\coordinate (k) at (-0.25,1);
\coordinate (l) at (0.25,1);
\coordinate (m) at (1.75,1);
\coordinate (n) at (4.25,1);
\coordinate (o) at (5.75,1);

\node[above] at (0,4)  {$I_{\n_1}:=I$};
\node at (0,3.5)  {$u_{\n_1}$};
\node[below] at (-3,3)  {$I_{\n_2}:=I:u_{\n_1}$};
\node[below] at (3,3)  {$I_{\n_3}:=I+Su_{\n_1}$};
\node[below] at (-5,2)  {$I_{\n_4}:=I_{\n_2}:u_{\n_2}$};
\node[below] at (-1.5,2)  {$I_{\n_5}:=I_{\n_2}+Su_{\n_2}$};
\node[below] at (1.5,2)  {$I_{\n_6}:=I_{\n_3}:u_{\n_3}$};
\node[below] at (5,2)  {$I_{\n_7}:=I_{\n_3}+Su_{\n_3}$};
\node at (-3,2.15)  {$u_{\n_2}$};
\node at (3,2.15)  {$u_{\n_3}$};

\node at (-5,1.15)  {$u_{\n_4}$};
\node at (-1,1.15)  {$u_{\n_5}$};
\node at (1,1.15)  {$u_{\n_6}$};
\node at (5,1.15)  {$u_{\n_7}$};

\node[below] at (-5.75,1) {$\vdots$};
\node[below] at (-4.25,1) {$\vdots$};
\node[below] at (-1.75,1) {$\vdots$};
\node[below] at (-0.25,1) {$\vdots$};
\node[below] at (0.25,1) {$\vdots$};
\node[below] at (1.75,1) {$\vdots$};
\node[below] at (4.25,1) {$\vdots$};
\node[below] at (5.75,1) {$\vdots$};

\draw [->] (a) -- (c);
\draw [->] (a) -- (b);
\draw [->] (b') -- (d);
\draw [->] (b') -- (e);
\draw [->] (c') -- (f);
\draw [->] (c') -- (g);
\draw [->] (d') -- (h);
\draw [->] (d') -- (i);
\draw [->] (e') -- (j);
\draw [->] (e') -- (k);
\draw [->] (f') -- (l);
\draw [->] (f') -- (m);
\draw [->] (g') -- (n);
\draw [->] (g') -- (o);
\end{tikzpicture}$$
$\T$ is called the \emph{ideal tree} of $I$ and the number of all pretty cleaner monomials $u_{\n_1},u_{\n_2},\ldots$ appeared in $\T$ is called the \emph{length of} $\T$. We denote the length of $\T$ by $l(\T)$.

We define the \emph{pretty $k$-cleanness length} of the pretty $k$-clean monomial ideal $I$ by $$l(I)=\min\{l(\T):\T\ \mbox{is an ideal tree of}\ I\}.$$

The following proposition gives an useful description of the structure of pretty clean filtrations.

\begin{prop}\label{pretty He-Po}
\cite[Proposition 10.1.]{HePo} Let $S=K[x_1,\ldots,x_n]$ be the polynomial ring, and $I\subset S$ a
monomial ideal. The following conditions are equivalent:
\begin{enumerate}[\upshape (a)]
  \item $S/I$ admits a multigraded prime filtration $\F:(0)=M_0\subset M_1\subset\ldots\subset M_{r-1}\subset M_r=S/I$ such that $M_i/M_{i-1}\cong S/P_i(-\a_i)$ for all $i$;
  \item there exists a chain of monomial ideals $I=I_0\subset I_1\subset\ldots\subset I_r=S$ and monomials $u_i$ of multidegree $\a_i$ such that $I_i=I_{i-1}+Su_i$ and $I_{i-1}:u_i=P_i$;
\end{enumerate}
\end{prop}

As an immediate consequence of the previous proposition we get

\begin{cor}\label{pretty irr}
Let $S=K[x_1,\ldots,x_n]$ be the polynomial ring, and $I\subset S$ a
monomial ideal. Let $S/I$ be pretty clean with the multigraded prime filtration $\F:I=I_0\subset I_1\subset\ldots\subset I_r=S$ such that $I_i/I_{i-1}\cong S/P_i(-\a_i)$ for all $i$. Set $I_i=\bigcap^r_{j=i+1}J_j$ for $i=0,\ldots,r$. Then $\ass(I_i)=\{P_{i+1},\ldots,P_r\}$ for all $i=0,\ldots,r$.
\end{cor}

Now we want to prove the main result of this section.

\begin{thm}\label{Equ. pretty k-clean}
Every pretty $k$-clean monomial ideal is pretty clean. Also, a pretty clean monomial ideal is pretty $k$-clean, for some $k\geq 0$.
\end{thm}
\begin{proof}
Suppose that $I$ is a pretty $k$-clean monomial ideal. We use induction on the pretty $k$-cleanness length of $I$. Assume that $I$ is not prime and there exists a pretty cleaner monomial $u\not\in I$ with $|\supp(u)|\leq k+1$ such that both $I:u$ and $I+Su$ are pretty $k$-clean. By induction, $I:u$ and $I+Su$ are pretty clean and there are pretty clean filtrations
$$\F_1:I+Su=J_0\subset J_1\subset\ldots\subset J_r=S$$
and
$$\F_2:0=\frac{L_0}{I:u}\subset \frac{L_1}{I:u}\subset\ldots\subset \frac{L_s}{I:u}=\frac{S}{I:u}$$
with $(L_i/I:u)/(L_{i-1}/I:u)\cong S/Q_i(-\a_i)$ where $Q_i$ are prime ideals.
It is known that the multiplication map $\varphi:S/I:u(-\a)\overset{.u}{\longrightarrow}I+Su/I$ is an isomorphism. Restricting $\varphi$ to $L_i/I:u$ yields a monomorphism $\varphi_i:L_i/I:u\overset{.u}{\longrightarrow}I+Su/I$. Set $H_i/I:=\varphi_i(L_i/I:u)$. Hence $H_i/I\cong (L_i/I:u)(-\a)$. It follows that
$$\frac{H_i}{H_{i-1}}\cong \frac{H_i/I}{H_{i-1}/I}\cong \frac{(L_i/I:u)(-\a)}{(L_{i-1}/I:u)(-\a)}\cong \frac{S}{Q_i}(-\a-\a_i).$$
Therefore we obtain the following prime filtration induced from $\F_2$:
$$\F_3:I=H_0\subset H_1\subset\ldots\subset H_s=I+Su.$$
By adding $\F_1$ to $\F_3$ we get the prime filtration
$$\F:I=H_0\subset H_1\subset\ldots\subset H_s=I+Su\subset J_1\subset\ldots\subset J_r=S.$$
Let $Q_i\in\Supp(\F_1)$ and $P_j\in\Supp(\F_2)$ with $P_j\subseteq Q_i$. By \cite[Corollary 3.6]{HePo}, $Q_i\in\ass(I+Su)$ and $P_j\in\ass(I:u)$. Since $u$ is a pretty cleaner we have $P_j=Q_i$. Therefore $I$ is pretty clean.

Conversely, let $I$ be a pretty clean monomial ideal. Then there is a pretty clean filtration
$$\F:(0)=M_0\subset M_1\subset \ldots\subset M_{r-1}\subset M_r=S/I$$
of $S/I$ with $M_i/M_{i-1}\cong S/P_i(-\a_i)$. If $I$ is a prime ideal then we have nothing to prove. Assume that $I$ is not a prime ideal. Since $I$ is pretty clean, by Proposition \ref{pretty He-Po}, there exists a chain of monomial ideals $I=I_0\subset I_1\subset\ldots\subset I_r=S$ and monomials $u_i$ of multidegree $\a_i$ such that $I_i=I_{i-1}+Su_i$ and $I_{i-1}:u_i=P_i$. It is clear that $I_1$ is pretty $k$-clean, where $|\supp(u_1)|\leq k+1$. By Corollary \ref{pretty irr}, $\ass(I_1)=\{P_2,\ldots,P_r\}$. It follows from $P_1\subset P_i\in \ass(I_1)$ that $P_1=P_i$. Hence, since $u_1$ is pretty cleaner, we obtain that $I$ is pretty $k$-clean.
\end{proof}

The following result is an improvement of \cite[Corollary 3.5.]{HePo} in the special case where $M$ is the quotient ring $S/I$.

\begin{thm}\label{pretty k-clean}
Let $I\subset S$ be a pretty $k$-clean monomial ideal. Then $I$ is $k$-clean if and only if $\ass(I)=\min(I)$.
\end{thm}
\begin{proof}
It follows from the definition.
\end{proof}

\begin{thm}\label{colon}
Let $I\subset S$ be pretty $k$-clean. Then for all monomial $u\in S$, $I:u$ is pretty $k$-clean.
\end{thm}
\begin{proof}
See the proof of Theorem 3.1. of \cite{Ra}.
\end{proof}

\begin{thm}
The radical of each pretty $k$-clean monomial ideal is pretty $k$-clean and so is $k$-clean.
\end{thm}
\begin{proof}
See the proof of Theorem 3.2. of \cite{Ra}.
\end{proof}

\begin{rem}
For some examples of pretty $k$-clean monomial ideals see \cite{Ra}.
\end{rem}

\section{$k$-decomposable multicomplexes}

The aim of this section is to extend the concept of $k$-decomposability to multicomplexes. We first define some notions. Let $\Gamma$ be a multicomplex and $\a\in\Gamma$. We define the \emph{star}, \emph{deletion} and \emph{link} of $\a$ in $\Gamma$, respectively, as follows:

$$\begin{array}{l}
  \s_\Gamma\a=\langle\b\in\F(\Gamma)|\a\preceq\b\rangle, \\
  \Gamma\backslash\a=\langle\b\in\F(\Gamma)|\a\npreceq\b\rangle, \\
  \lk_\Gamma\a=\langle\b-\a:\b\in\F(\Gamma)\ \mbox{and}\ \a\preceq\b\rangle.
\end{array}
$$
For the multicomplexes $\Gamma_1,\Gamma_2\subset\NN^n_\infty$, the \emph{join} of $\Gamma_1$ and $\Gamma_2$ is defined to be $$\Gamma_1\cdot\Gamma_2=\{\a+\b:\a\in\Gamma_1,\b\in\Gamma_2\}.$$

One can easily check that
$$\begin{array}{l}
  \s_\Gamma\a=\langle\a\rangle\cdot\lk_\Gamma\a, \\
  \s_\Gamma\a=\{\b\in\Gamma|\a\vee\b\in\Gamma\}\ \mbox{and} \\
  \lk_\Gamma\a=\{\a\vee\b-\a:\a\vee\b\in\Gamma\}.
\end{array}$$

If $\{\a_1,\ldots,\a_r\}\subset \NN^n_\infty$, then
$$\Gamma\backslash\{\a_1,\ldots,\a_r\}=\langle \b\in\F(\Gamma):\a_i\npreceq\b\ \mbox{for all}\ i\rangle=\bigcap^r_{i=1}\Gamma\backslash\a_i.$$

\begin{exam}
Let $\Gamma=\langle(2,\infty),(3,0)\rangle$. Then $$\F(\Gamma)=\{(0,\infty),(1,\infty),(2,\infty),(3,0)\}.$$

For $\a=(2,1)$ we have
$$\begin{array}{l}
    \s_\Gamma\a=\langle(2,\infty)\rangle, \\
    \Gamma\backslash\a=\langle(0,\infty),(1,\infty),(3,0)\rangle, \\
    \lk_\Gamma\a=\langle(0,\infty)\rangle.
  \end{array}
$$
\end{exam}

\begin{defn}\label{shedding face}
Let $\Gamma$ be a ($d-1$)-dimensional multicomplex and let $0\leq k\leq d-1$. An element $\a\in\Gamma\cap\NN^n$ with $|\fpt^\ast(\a)|\leq k+1$ is called \emph{shedding face} if it satisfies the following conditions:
\begin{enumerate}[\upshape (i)]
  \item for all $\b\in \F(\s_\Gamma(\a))$, $\langle\b\rangle\backslash(\Gamma\backslash\a)$ is a Stanley set of degree $\a$;
  \item for every $\b\in \F(\s_\Gamma(\a))$ and every $\c\in \F(\Gamma\backslash\a)$ if $\fpt(\b)\subseteq\fpt(\c)$ then $\fpt(\b)=\fpt(\c)$.
\end{enumerate}
\end{defn}

\begin{defn}\label{k-decom}
Let $\Gamma$ be a ($d-1$)-dimensional multicomplex and let $0\leq k\leq d-1$. We say that $\Gamma$ is \emph{$k$-decomposable} if it has only one facet or there exists a shedding face $\a\in\Gamma$ with $|\fpt^\ast(\a)|\leq k+1$ such that both $\lk_\Gamma(\a)$ and $\Gamma\backslash\a$ are $k$-decomposable.
\end{defn}

\begin{rem}\label{one facet}
Note that for a multicomplex $\Gamma$ with $\F(\Gamma)=\{\a\}$ one has $\a\in\{0,\infty\}^n$ (see \cite[Corollary 9.11]{HePo}).
\end{rem}

Now we discuss some structural properties of $k$-decomposable multicomplexes.

\begin{thm}\label{link k-decom}
Let $\Gamma$ be a $k$-decomposable multicomplex. Then for all $\a\in\Gamma$, $\lk_\Gamma\a$ is $k$-decomposable.
\end{thm}
\begin{proof}
If $\Gamma$ has just one facet then we have no thing to prove. Suppose that $|\F(\Gamma)|>1$ and there is a shedding face $\b\in\Gamma$ with $|\fpt^\ast(\b)|\leq k+1$.

Case 1. Let $\b\preceq\a$ and $\a\vee\b\in\Gamma$. Then $\lk_\Gamma\a=\lk_{\lk_\Gamma\b}(\a-\b)$. Since $|\F(\lk_\Gamma\b)|\leq |\F(\Gamma)|$, it follows from induction hypothesis that $\lk_{\lk_\Gamma\b}(\a-\b)$ is $k$-decomposable.

Case 2. Let $\b\npreceq\a$ and $\a\vee\b\in\Gamma$. Then
$$\lk_\Gamma\a\backslash(\a\vee\b-\a)=\lk_{\Gamma\backslash\b}\a$$
and
$$\lk_{\lk_\Gamma\a}(\a\vee\b-\a)=\lk_\Gamma(\a\vee\b)=\lk_{\lk_\Gamma\b}(\a\vee\b-\b).$$
Now, since $|\F(\lk_\Gamma\b)|\leq\F(\Gamma)|$ and $|\F(\Gamma\backslash\b)|\leq |\F(\Gamma)|$ we conclude that $\lk_{\lk_\Gamma\a}(\a\vee\b-\a)$ and $\lk_\Gamma\a\backslash(\a\vee\b-\a)$ are $k$-decomposable, by induction hypothesis. Now we show that $\a\vee\b-\a$ is a shedding face of $\lk_\Gamma\a$.

Let $\c\in\F(\s_{\lk_\Gamma\a}(\a\vee\b-\a))$. Hence, since $\s_{\lk_\Gamma\a}(\a\vee\b-\a)=\lk_{\s_\Gamma\b}(\a)$ we get $\c+\a\in\F(\s_{\s_\Gamma\b}(\a))$. Thus $\c+\a\in\F(\s_\Gamma\b)$. It follows that there is $m\in\{0,\infty\}^n$ such that $\langle\c+\a\rangle\backslash(\Gamma\backslash\b)=\b+\langle m\rangle$. This implies that
$$\langle\c\rangle\backslash (\lk_\Gamma\a\backslash(\a\vee\b-\a))=\langle\c\rangle\backslash (\lk_{\Gamma\backslash\b}\a)=\a\vee\b-\a+\langle m\rangle.$$

Let $\u\in\F(\s_{\lk_\Gamma\a}(\a\vee\b-\a))$ and $\v\in\F((\lk_\Gamma\a)\backslash(\a\vee\b-\a))$ with $\fpt(\u)\subseteq\fpt(\v)$. Then we have $\u+\a\in\F(\s_\Gamma\b)$, $\v+\a\in\F(\Gamma\backslash\b)$ and $\fpt(\u+\a)\subseteq\fpt(\v+\a)$. Because $\b$ is a shedding face of $\Gamma$ we get $\fpt(\u+\a)=\fpt(\v+\a)$. It follows that $\fpt(\u)=\fpt(\v)$.

Case 3. Let $\a\vee\b\not\in\Gamma$. Then $\lk_\Gamma\a=\lk_{\Gamma\backslash\b}\a$.
Since $|\F(\Gamma\backslash\b)|\leq |\F(\Gamma)|$, it follows from induction hypothesis that $\lk_\Gamma\a$ is $k$-decomposable.

\end{proof}

\begin{thm}\label{one max-facet}
Let $\Gamma\in\NN^n_\infty$ be a multicomplex which has just one maximal facet $\b$. Then $\Gamma$ is $k$-decomposable if and only if $|\fpt^\ast(\b)|\leq k+1$.
\end{thm}
\begin{proof}
``Only if part'': Let $\Gamma$ be $k$-decomposable. If $\Gamma$ has only one facet then the assertion follows from Remark \ref{one facet}. Suppose that $|\F(\Gamma)|>1$ and let $\a$ be a shedding face of $\Gamma$ with $|\fpt^\ast(\a)|\leq k+1$ such that $\lk_\Gamma\a$ and $\Gamma\backslash\a$ are $k$-decomposable. Since $\b\in\F(\s_\Gamma\a)$, there exists $m\in\{0,\infty\}^n$ such that $\langle\b\rangle\backslash(\Gamma\backslash\a)=\a+\langle m\rangle$. Note that $\infpt(\b)=\infpt(m)$.

Let $0<\b(i)<\infty$ for some $i$. If $\a(i)=0$ then since $\b\in\a+\langle m\rangle$ we have $0<m(i)<\infty$, a contradiction. Therefore $\a_(i)\neq 0$. This implies that $\fpt^\ast(\b)\subseteq\fpt^\ast(\a)$. Hence $|\fpt^\ast(\b)|\leq k+1$.

``If part'': If $|\fpt^\ast(\b)|=0$ then $\b\in\{0,\infty\}^n$ and so $\Gamma$ has just one facet. Hence $\Gamma$ is $k$-decomposable. Suppose that $|\fpt^\ast(\b)|>0$. We show that $\a$ with
$$\a(i)=\left\{
    \begin{array}{ll}
      \b(i), & \b(i)\neq\infty \\
      0, & \hbox{otherwise}
    \end{array}
  \right.
$$

is a shedding face of $\Gamma$.

Since $\F(\lk_\Gamma(\a))=\{\b-\a\}$, it follows that $\lk_\Gamma\a$ is $k$-decomposable. We have $\langle\b\rangle\backslash(\Gamma\backslash\a)=\a+\langle m\rangle$ where $\infpt(m)=\infpt(\b)$. Since for all $\c\in\F(\Gamma)$, $\infpt(\c)=\infpt(\b)$ it follows that $\fpt(\c)=\fpt(\b)$ and so the condition (ii) of Definition \ref{shedding face} holds. It remains to show that $\Gamma\backslash\a$ is $k$-decomposable.

Let $0<\b(i)<\infty$. Set
$$\c(j)=\left\{
    \begin{array}{ll}
      \b(j), & j\neq i \\
      \b(i)-1, & j=i
    \end{array}
  \right.
$$
In a similar way to $\a$ for $\Gamma$, we show that $\c$ is a shedding face of $\Gamma\backslash\a$. The proof is completed inductively.

Consequently, $\Gamma$ is $k$-decomposable.
\end{proof}

Two multicomplexes $\Gamma_1,\Gamma_2\subset\NN^n_\infty$ are called \emph{disjoint} whenever there exists $1<m<n$ such that $\a\in\Gamma_1$ (resp. $\a\in\Gamma_2$) implies $\a(i)=0$ for $i>m$ (resp. $\a(i)=0$ for $i\leq m$).

\begin{thm}\label{join}
Let $\Gamma_1$ and $\Gamma_2$ be two disjoint multicomplexes. If $\Gamma_1\cdot\Gamma_2$ is $k$-decomposable then $\Gamma_1$ and $\Gamma_2$ are $k$-decomposable. The converse holds, if in addition, $\F(\Gamma_2)\subset\{0,\infty\}^n$.
\end{thm}
\begin{proof}
Note that $\Gamma=\Gamma_1\cdot\Gamma_2$ has one facet if and only if $\Gamma_1$ and $\Gamma_2$ have one facet. So assume that $|\F(\Gamma_1)|>1$ or $|\F(\Gamma_2)|>1$.

For every face $\a\in\Gamma$ we have
\begin{equation}\label{1}
    \lk_\Gamma(\a)=\lk_{\Gamma_1}(\a_1)\cdot\lk_{\Gamma_2}(\a_2)
\end{equation}
and
\begin{equation}\label{2}
\Gamma\backslash\a=(\Gamma_1\backslash\a_1)\cdot\Gamma_2\cup\Gamma_1\cdot(\Gamma_2\backslash\a_2)
\end{equation}
where $\a_1\in\Gamma_1$, $\a_2\in\Gamma_2$ and $\a=\a_1+\a_2$.

``Only if part'': Let $\Gamma_1\cdot\Gamma_2$ be $k$-decomposable with shedding face $\a=\a_1+\a_2$ where $\a_i\in\Gamma_i$. We want to show that $\Gamma_i$ is $k$-decomposable with shedding face $\a_i$. We may assume that $\s_{\Gamma_1}\a_1\neq\Gamma_1$. Since $\lk_\Gamma\a_1=\lk_{\Gamma_1}\a_1\cdot\Gamma_2$, we get $\lk_{\Gamma_1}\a_1$ and $\Gamma_2$ are $k$-decomposable, by induction. On the other hand, $\Gamma\backslash\a$ is $k$-decomposable. Hence $\lk_{\Gamma\backslash\a}\a_2=\Gamma_1\backslash\a_1\cdot\lk_{\Gamma_2}\a_2$ is $k$-decomposable, by Theorem \ref{link k-decom}. Thus $\Gamma_1\backslash\a_1$ is $k$-decomposable, by induction.

Let $\b_1\in\F(\s_{\Gamma_1}\a_1)$. Choose a facet $\b_2\in\F(\s_{\Gamma_2}\a_2)$ and set $\b=\b_1+\b_2$. Then $\b\in\F(\s_\Gamma\a)$ and
\begin{equation}\label{3}
    \begin{array}{rl}
\langle\b\rangle\backslash(\Gamma\backslash\a)=& [\langle\b\rangle\backslash(\Gamma_1\backslash\a_1\cdot\Gamma_2)]\cap[\langle\b\rangle\backslash(\Gamma_1\cdot\Gamma_2\backslash\a_2)] \\
    = & [\langle\b_1\rangle\backslash(\Gamma_1\backslash\a_1)\cdot\langle\b_2\rangle]\cap[\langle\b_1\rangle\cdot\langle\b_2\rangle\backslash(\Gamma_2\backslash\a_2)]\\
    = & \langle\b_1\rangle\backslash(\Gamma_1\backslash\a_1)\cdot\langle\b_2\rangle\backslash(\Gamma_2\backslash\a_2).
  \end{array}
\end{equation}
On the other hand, $\langle\b\rangle\backslash(\Gamma\backslash\a)=\a+\langle m\rangle$ where $m(i)\in\{0,\infty\}$. Let $\a=\a_1+\a_2$ and $m=m_1+m_2$ where $\a_i,m_i\in\Gamma_i$. We conclude from (\ref{3}) that $\langle\b_1\rangle\backslash(\Gamma_1\backslash\a_1)=\a_1+\langle m_1\rangle$.

Let $\b_1\in\F(\s_{\Gamma_1}\a_1)$ and $\c_2\in\F(\Gamma_1\backslash\a_1)$ with $\fpt(\b_1)\subseteq\fpt(\c_1)$. Choose $\b_2\in\F(\s_{\Gamma_2}\a_2)$ and $\c_2\in\F(\Gamma_2\backslash\a_2)$ with $\fpt(\b_2)\subseteq\fpt(\c_2)$. Then there is $\c'_2\in\M(\Gamma_2)$ such that $\c_2\preceq\c'_2$. It follows that $\b_1+\b_2\in\F(\s_\Gamma\a)$ and $\c_1+\c'_2\in\F(\Gamma\backslash\a)$. Note that $\fpt(\c_2)=\fpt(\c'_2)$. Therefore $\fpt(\b_1+\b_2)=\fpt(\c_1+\c'_2)$. In particular, $\fpt(\b_1+\b_2)=\fpt(\c_1+\c_2)$. It follows that $\fpt(\b_1)=\fpt(\c_1)$. Therefore $\a_1$ is a shedding face of $\Gamma_1$.

``If part'': Let $\Gamma_1$ and $\Gamma_2$ be $k$-decomposable with shedding faces $\a_1$ and $\a_2$, respectively, and let $\F(\Gamma_2)\subset\{0,\infty\}^n$. We claim that $\a_1$ is a shedding face of $\Gamma$.

It follows from relations (\ref{1}) and (\ref{2}) that
$$\lk_\Gamma(\a_1)=\lk_{\Gamma_1}(\a_1)\cdot\Gamma_2,\quad \Gamma\backslash\a_1=(\Gamma_1\backslash\a_1)\cdot\Gamma_2.$$
By induction hypothesis, $\lk_\Gamma(\a_1)$ and $\Gamma\backslash\a_1$ are $k$-decomposable.

Let $\b=\b_1+\b_2\in \F(\s_\Gamma(\a_1))$ where $\b_i\in\Gamma_i$. Then $$\langle\b\rangle\backslash(\Gamma\backslash\a_1)=[\langle\b_1\rangle\backslash(\Gamma_1\backslash\a_1)]\cdot\langle\b_2\rangle=(\a_1+\langle m\rangle)\cdot\langle\b_2\rangle=\a_1+\langle m+\b_2\rangle$$ is a Stanley set.

Let $\b=\b_1+\b_2\in \F(\s_\Gamma\a_1)$ and $\c=\c_1+\c_2\in \F(\Gamma\backslash\a_1)$ with $\b_i,\c_i\in\Gamma_i$. Suppose that $\fpt(\b)\subseteq\fpt(\c)$. Then $\fpt(\b_i)\subseteq\fpt(\c_i)$, for $i=1,2$. Since $\b_2$ and $\c_2$ are facets of $\Gamma_2$ and, moreover, $\F(\Gamma_2)\subset\{0,\infty\}^n$, we have $\infpt(\c_2)=\infpt(\b_2)$, by definition. Thus $\fpt(\b_2)=\fpt(\c_2)$. On the other hand, by $k$-decomposability of $\Gamma_1$, $\fpt(\b_1)=\fpt(\c_1)$. Therefore $\fpt(\b)=\fpt(\c)$, as desired.
\end{proof}

Now we come to the main result of this section.

\begin{thm}\label{shell k-decom}
Every $k$-decomposable multicomplex $\Gamma$ is shellable. Also, every shellable multicomplex is $k$-decomposable for some $k\geq 0$.
\end{thm}
\begin{proof}
Let $\Gamma$ be $k$-decomposable. If $\Gamma$ has only one facet then we are done. Suppose that $|\F(\Gamma)|>1$. Let $\a\in\Gamma$ be a shedding face of $\Gamma$ with $|\fpt^\ast(\a)|\leq k+1$ such that $\lk_\Gamma\a$ and $\Gamma\backslash\a$ are $k$-decomposable. By induction, $\lk_\Gamma(\a)$ and $\Gamma\backslash\a$ are shellable. Let $\a_1,\ldots,\a_t$ and $\a_{t+1}-\a,\ldots,\a_r-\a$ be, respectively, shelling orders of $\Gamma\backslash\a$ and $\lk_\Gamma\a$. It is easy to check that $\a_{t+1},\ldots,\a_r$ is a shelling order of $\s_\Gamma\a$. We claim that $\a_1,\ldots,\a_r$ is a shelling order of $\Gamma$.

We want to show that $S_i=\langle \a_i\rangle\backslash\langle \a_1,\ldots,\a_{i-1}\rangle$ is a Stanley set, for al $i$. The case $i\leq t$ is clear. Suppose that $i>t$. Clearly,
$$\langle \a_i\rangle\backslash\langle \a_1,\ldots,\a_{i-1}\rangle=(\langle \a_i\rangle\backslash\langle \a_1,\ldots,\a_t\rangle)\cap(\langle \a_i\rangle\backslash\langle \a_{t+1},\ldots,\a_{i-1}\rangle).$$
Because $\Gamma$ is $k$-decomposable we have $\langle\a_i\rangle\backslash\langle \a_1,\ldots,\a_t\rangle=\a+\langle m\rangle$ where $m\in\{0,\infty\}^n$. Moreover, $\s_\Gamma(\a)$ is shellable and hence there exist $\a'\in\NN^n$ with $|\fpt^*(\a')|\leq k+1$ and $m'\in\{0,\infty\}^n$ such that
\begin{center}
$\langle\a_i\rangle\backslash\langle \a_{t+1},\ldots,\a_{i-1}\rangle=\a'+\langle m'\rangle$.
\end{center}
It is clear that $m=m'$. Therefore $\langle \a_i\rangle\backslash\langle \a_1,\ldots,\a_{i-1}\rangle=\a\vee\a'+\langle m\rangle$.

Let $S^\ast_i\subseteq S^\ast_j$. If $i,j\leq t$ or $t\leq i,j$ then we are done, because $\s_\Gamma(\a)$ and $\Gamma\backslash\a$ are shellable. Suppose that $i\leq t<j$. Since $\infpt(\a_j)=\infpt(S^\ast_j)$ and $\infpt(\a_i)=\infpt(S^\ast_i)$ we have $\fpt(\a_j)\subseteq\fpt(\a_i)$. But $\fpt(\a_j)=\fpt(\a_i)$, because $\a_j\in\F(\s_\Gamma(\a))$ and $\a_i\in\F(\Gamma\backslash\a)$. Consequently, $\infpt(\a_j)=\infpt(\a_i)$ and so $S^\ast_i=S^\ast_j$, as desired.

For the second part of theorem, suppose that $\Gamma$ is shellable with the shelling $\a_1,\ldots,\a_r$. If $r=1$ then $\Gamma$ is $k$-decomposable for some $k\geq 0$. So assume that $r>1$. We proceed by induction on the number of facets of $\Gamma$. Since $S_r=\langle\a_r\rangle\backslash\langle\a_1,\ldots,\a_{r-1}\rangle$ is a Stanley set, so there exists $\a\in\NN^n $ and $m\in\{0,\infty\}^n$ such that $S_r=\a+\langle m\rangle$. Let $|\fpt^*(\a)|\leq k+1$ for some $k\geq 0$. It is clear that $\s_\Gamma(\a)=\langle\a_r\rangle$ and $\Gamma\backslash\a=\langle\a_1,\ldots,\a_{r-1}\rangle$. By induction hypothesis, $\Gamma\backslash\a$ is $k'$-decomposable for some $k'\geq 0$. Assume that $k\geq k'$. If we show that $\a$ satisfies the condition (ii) of Definition \ref{shedding face} then we have shown that $\a$ is a shedding face of $\Gamma$.

Let $i<r$ and $\fpt(\a_r)\subseteq\fpt(\a_i)$. Then $\infpt(\a_i)\subseteq\infpt(\a_r)$ and so $S^\ast_i\subseteq S^\ast_r$. It follows that $S^\ast_i=S^\ast_r$. This implies that $\fpt(\a_r)=\fpt(\a_i)$, as the desired.
\end{proof}

For $F\subset [n]$ define $\a_F\in\NN^n_\infty$ by $\a_F(i)=\infty$ if $i\in F$ and $\a_F(i)=0$, otherwise. Also, for $\a\in\{0,\infty\}^n$ set $F_\a=\{i\in[n]:\a(i)=\infty\}$. The next result shows that our definition of $k$-decomposability of multicomplexes
extends the concept of $k$-decomposability of simplicial complexes defined in \cite{BiPr,Wo}.

\begin{prop}\label{Sim-multi k-decom}
Let $\Del$ be a simplicial complex with facets $F_1,\ldots,F_r$, and $\Gamma$ be the multicomplex with the facets $\a_{F_1},\ldots,\a_{F_r}$. Then $\Del$ is $k$-decomposable if and only if $\Gamma$ is $k$-decomposable.
\end{prop}
\begin{proof}
``Only if part'': We use induction on the number of the facets of $\Del$. Let $\Del$ be $k$-decomposable with shedding face $\sigma\in\Del$. We claim that $e_\sigma=\sum_{i\in F}e_i$ is a shedding face of $\Gamma$ where $e_i$ denotes the $i$th standard unit vector in $\NN^n$. Clearly, $|\fpt^\ast(e_\sigma)|\leq k+1$. Note that
$$\lk_\Gamma e_\sigma=\langle \a_F:F\in\F(\lk_\Del\sigma)\rangle$$
and
$$\Gamma\backslash e_\sigma=\langle\a_F:F\in\F(\Del\backslash\sigma)\rangle.$$
By induction, $\lk_\Gamma e_\sigma$ and $\Gamma\backslash e_\sigma$ are $k$-decomposable.

Let $\a_F\in \F(\s_\Gamma e_\sigma)$. Then $$\langle\a_F\rangle\backslash(\Gamma\backslash e_\sigma)=\{u\in\Gamma:e_\sigma\preceq u\preceq\a_F\}=e_\sigma+\langle\a_F\rangle.$$
Therefore $\langle\a_F\rangle\backslash(\Gamma\backslash e_\sigma)$ is a Stanley set of degree $e_\sigma$.

Consider $\b\in\F(\s_\Gamma e_\sigma)$ and $\c\in\F(\Gamma\backslash e_\sigma)$ with $\fpt(\b)\subseteq\fpt(\c)$. If $\fpt(\b)\neq\fpt(\c)$ then $F_\c\subsetneqq F_\b$ and so there exists $x\in\sigma$ such that $x\in F_\b\backslash F_\c$. Particularly, $F_\b\backslash x$ is a facet of $\s_\Del\sigma\backslash\sigma$ and $\Del\backslash\sigma$. This contradicts the assumption that $\sigma$ is a shedding face of $\Del$. Therefore $\fpt(\b)=\fpt(\c)$.

``If part'': If $r=1$ then we are done. Assume that $r>1$ and suppose that $\Gamma$ is $k$-decomposable and $\a\in\NN^n$ is a shedding face of $\Gamma$ with $|\fpt^\ast(\a)|\leq k+1$.

Set $\sigma=\fpt^\ast(\a)$. Since $\lk_\Gamma e_\sigma=\lk_\Gamma\a$ and $\Gamma\backslash\ e_\sigma=\Gamma\backslash\a$ thus $\lk_\Gamma e_\sigma$ and $\Gamma\backslash e_\sigma$ are $k$-decomposable. Hence by induction hypothesis, $\lk_\Del\sigma$ and $\Del\backslash\sigma$ are $k$-decomposable. It remains to show that $\sigma$ satisfies the exchange property. Let $F$ be a facet of both $\s_\Del\sigma\backslash\sigma$ and $\Del\backslash\sigma$. Then there exists a facet $G\in\s_\Del\sigma$ and $x\in\sigma$ such that $F=G\backslash x$. Clearly, $\infpt(\a_F)\subsetneqq\infpt(\a_G)$. It follows that $\fpt(\a_G)\subsetneqq\fpt(\a_F)$. This is a contradiction, because $\a_G\in \F(\s_\Gamma e_\sigma)$ and $\a_F\in \F(\Gamma\backslash e_\sigma)$. Therefore $\sigma$ is a shedding face of $\Del$.
\end{proof}

For the simplicial complexes
$\Del_1$ and $\Del_2$ defined on disjoint vertex sets, the \emph{join}
of $\Del_1$ and $\Del_2$ is $\Del_1.\Del_2=\{\sigma\cup\tau:\
\sigma\in\Del_1,\tau\in\Del_2\}$.

Theorem \ref{join} together with Proposition \ref{Sim-multi k-decom} now yields

\begin{cor}\label{join sim. com}
Let $\Del_1$ and $\Del_2$ be simplicial complexes on disjoint vertex sets. Then $\Del_1$ and $\Del_2$ are $k$-decomposable if and only if $\Del_1\cdot\Del_2$ is $k$-decomposable.
\end{cor}

\section{The main results}

In this section we present the main results of the paper. For the proof of the first main theorem we need the following lemma whose proof is easy and we leave without proof.

\begin{lem}\label{ideal of lk-dl-st}
Let $\Gamma$ be a multicomplex and $\a\in\Gamma$. Then
  %\item $I(\s_\Gamma\a)=(x^\b:x^{\a\vee\b}\in I(\Gamma))$;
\begin{center}
$I(\Gamma\backslash\a)=I(\Gamma)+Sx^\a$ and $I(\lk_\Gamma\a)=I(\Gamma):x^\a$.
\end{center}
It follows that for a monomial ideal $I$ and a monomial $x^\a\in I$ we have
\begin{center}
    $\Gamma(I:x^\a)=\lk_{\Gamma(I)}\a$ and $\Gamma(I+Sx^\a)=\Gamma(I)\backslash\a$.
\end{center}
\end{lem}

We are prepare to prove the first main result of this section which is an improvement of \cite[Theorem 10.5.]{HePo}.

\begin{thm}\label{k-decom pretty}
Let $\Gamma$ be a multicomplex. Then $\Gamma$ is $k$-decomposable if and only if $I(\Gamma)$ is pretty $k$-clean.
\end{thm}
\begin{proof}
``Only if part'': Let $\F(\Gamma)=\{\a_1,\ldots,\a_r\}$. If $r=1$, then $I(\Gamma)$ is a prime ideal and so we have no thing to prove. So suppose that $r>1$. Let $\Gamma$ be $k$-decomposable with shedding face $\a$. It follows from induction hypothesis and Lemma \ref{ideal of lk-dl-st} that $I(\Gamma):x^\a$ and $I(\Gamma)+Sx^\a$ are pretty $k$-clean.

Let $P\in \ass(I(\lk_\Gamma\a))$ and $Q\in \ass(I(\Gamma\backslash\a))$ with $P\subseteq Q$. Hence there exist $\b\in\F(\lk_\Gamma\a)$ and $\c\in\F(\Gamma\backslash\a)$ such that $P=\sqrt{I(\langle\b\rangle)}$ and $Q=\sqrt{I(\langle\c\rangle)}$ (see the proof of \cite[Theorem 10.1]{HePo}). Then $\fpt(\b+\a)\subseteq\fpt(\b)\subseteq\fpt(\c)$ where $\b+\a$ is a facet of $\s_\Gamma\a$. By $k$-decomposability of $\Gamma$, we have $\fpt(\b+\a)=\fpt(\c)$. In particular, $\fpt(\b)=\fpt(\c)$ and so $P=Q$. Therefore $x^\a$ is a pretty cleaner of $I(\Gamma)$, as desired.

``If part'': Let $I(\Gamma)$ be pretty $k$-clean. If $\Gamma$ has just one facet then we are done. Suppose that $\Gamma$ has more than one facet. Then $I(\Gamma)$ is not prime. Thus there exists a pretty cleaner monomial $x^\a\not\in I(\Gamma)$ with $|\supp(x^\a)|\leq k+1$ such that $I(\Gamma):x^\a$ and $I(\Gamma)+Sx^\a$ are pretty $k$-clean. It follows from Lemma \ref{ideal of lk-dl-st} and induction hypothesis that $\lk_\Gamma\a$ and $\Gamma\backslash\a$ are $k$-decomposable.

Let $\b\in\F(\s_\Gamma\a)$. We want to show that $\langle\b\rangle\backslash(\Gamma\backslash\a)$ is a Stanley set. By the proof of Theorem 10.6 and the discussion at the end of Section 6 from \cite{HePo}, $\langle\b\rangle\backslash(\Gamma\backslash\a)$ is a Stanley set if and only if $I(\Gamma\backslash\a)/I(\langle\Gamma\backslash\a,\b\rangle)$ is a cyclic quotient. Therefore it is enough to show that $I(\Gamma\backslash\a)/I(\langle\Gamma\backslash\a,\b\rangle)$ is a cyclic quotient. It is easy to check that $\langle\Gamma\backslash\a,\b\rangle$ is $k$-decomposable with shedding face $\a$ and so by the only if part, $I(\langle\Gamma\backslash\a,\b\rangle)$ is pretty $k$-clean. It follows from Theorem \ref{Equ. pretty k-clean} that $I(\langle\Gamma\backslash\a,\b\rangle)$ is pretty clean and so by \cite[Theorem 10.5.]{HePo}, $I(\Gamma\backslash\a)/I(\langle\Gamma\backslash\a,\b\rangle)\cong S/P$ is a cyclic quotient where $P=(x_i:i\in\fpt(\b))$, as desired.

Let $\b\in\F(\s_\Gamma\a)$ and $\c\in\F(\Gamma\backslash\a)$ with $\fpt(\b)\subset\fpt(\c)$. Set
$$\begin{array}{ccc}
    \b'(i)=\left\{
           \begin{array}{ll}
             0 & i\in\fpt(\b) \\
             \infty & \mbox{otherwise}
           \end{array}
         \right. & \mbox{and}&
\c'(i)=\left\{
           \begin{array}{ll}
             0 & i\in\fpt(\c) \\
             \infty & \mbox{otherwise}.
           \end{array}
         \right.
  \end{array}
$$
Clearly, $\b'\in\F(\s_\Gamma\a)$ and $\c'\in\F(\Gamma\backslash\a)$ with $\fpt(\b')\subset\fpt(\c')$. We have $\b'\vee\a-\a\in\F(\lk_\Gamma\a)$. Let $P=\sqrt{I(\langle \b'\vee\a-\a\rangle)}$ and $Q=\sqrt{I(\langle \c'\rangle)}$. Then $P\in \ass(I(\lk_\Gamma\a))$ and $Q\in\ass(I(\Gamma\backslash\a))$. Since $\fpt(\b'\vee\a-\a)=\fpt(\b')\subset\fpt(\c')$, we have $P\subseteq Q$. It follows that $P=Q$ and so $\fpt(\b)=\fpt(\b')=\fpt(\c')=\fpt(\c)$.

Therefore $\a$ is a shedding face of $\Gamma$.
\end{proof}

\begin{rem}
Note that Theorems \ref{link k-decom} and \ref{shell k-decom} are, respectively, combinatorial translations of Theorems \ref{colon} and \ref{Equ. pretty k-clean}.
\end{rem}

%For $\a\in\NN^n_\infty$ set $\supp(\a)=\{i\in [n]:\a(i)\neq 0\}$.

%\begin{thm}\label{rev k-decom pretty}
%Suppose that the multicomplex $\Gamma$ satisfies one of the following property:
%\begin{enumerate}[\upshape (i)]
%  \item $\M(\Gamma)\subseteq \{0,\infty\}^n$;
%  \item for all $\a\in\M(\Gamma)$, $|\supp(\a)|\leq k+1$.
%\end{enumerate}
%If $I(\Gamma)$ is pretty $k$-clean then $\Gamma$ is $k$-decomposable.
%\end{thm}
%\begin{proof}
%\end{proof}

\begin{rem}\label{prt not k-prt}
Consider the simplicial complex $$\Del=\langle 124,125,126,135,136,145,236,245,256,345,346\rangle$$
on $[6]$. It was shown in \cite{Si} that $\Del$ is shellable but not vertex-decomposable. It follows from Proposition \ref{Sim-multi k-decom} that the multicomplex $\Gamma$ with $\F(\Gamma)=\{\a_F:F\in\F(\Del)\}$ is shellable but not $0$-decomposable. This means that a pretty $k$-clean ideal need not be pretty $k'$-clean for $k>k'$. To see more examples of shellable simplicial complexes which are not vertex-decomposable we refer the reader to \cite{Ha,MoTa}.
\end{rem}

Let $I\subset S$ be a monomial ideal. We denote by $\Gamma$ and $\Gamma^p$ the multicomplexes associated to $I$ and $I^p$, respectively. Soleyman Jahan \cite{So} showed that there is a bijection between the facets of $\Gamma$ and the facets of $\Gamma^p$. We recall some notions of the construction of $\Gamma^p$ from \cite{So}. Let $I\subset S$ be minimally generated by $u_1,\ldots,u_r$ and let $D\subset [n]$ be the set of elements $i\in [n]$ such that $x_i$ divides $u_j$ for at least one $j=1,\ldots,r$. Then we set
$$t_i=\max\{s:x^s_i\ \mbox{divides}\ u_j\ \mbox{at least for one}\ j\in [m]\}$$
if $i\in D$ and $t_i=1$, otherwise. Moreover we set $t=\sum^n_{i=1}t_i$. For every $\a\in\F(\Gamma)$, the facet $\bar{\a}\in\F(\Gamma^p)$ is defined as follows: if $\a(i)=\infty$ then set $\bar{\a}(ij)=\infty$ for all $1\leq j\leq t_i$, and if $\a(i)<t_i$ then set
$$\bar{\a}(ij)=\left\{
    \begin{array}{ll}
      0 & \mbox{if}\ j=\a(i)+1 \\
      \infty & \mbox{otherwise}.
    \end{array}
  \right.
$$
It was shown in \cite[Proposition 3.8.]{So} that the map
$$\begin{array}{rc}
    \beta: & \F(\Gamma)\longrightarrow \F(\Gamma^p)\\
     & \a\longmapsto\bar{\a}
  \end{array}
$$
is a bijection.

\begin{thm}\label{polar k-decom}
Let $\Gamma$ be the multicomplex associated to a monomial ideal $I$. Then $\Gamma$ is $k$-decomposable if and only if $\Gamma^p$ is $k$-decomposable.
\end{thm}
\begin{proof}
``Only if part'': If $|\F(\Gamma)|=1$ then we have no thing to prove. Assume that $|\F(\Gamma)|>1$. Hence there exists a shedding face $\a\in\Gamma$ such that $\lk_\Gamma\a$ and $\Gamma\backslash\a$ are $k$-decomposable. We define $\a'\in\Gamma^p$ as follows: if $\a(i)=0$ then we set $\a'(ij)=0$ for all $1\leq j\leq t_i$ and if $\a(i)\neq 0$ then we set
$$\a'(ij)=\left\{
            \begin{array}{ll}
              1 & \mbox{if}\ 1\leq j\leq\a(i) \\
              0 & \mbox{otherwise.}
            \end{array}
          \right.
$$
It is easy to check that
\begin{center}
    $\lk_{\Gamma^p}\a'=(\lk_\Gamma\a)^p$ and $\Gamma^p\backslash\a'=(\Gamma\backslash\a)^p$.
\end{center}
By induction hypothesis, $\lk_{\Gamma^p}\a'$ and $\Gamma^p\backslash\a'$ are $k$-decomposable. Now we show that $\a'$ is a shedding face of $\Gamma^p$.

Let $\bar{\b}\in\F(\s_{\Gamma^p}\a')$. Then
$$\langle\bar{\b}\rangle\backslash(\Gamma^p\backslash\a')=\{u\in\Gamma^p:\a'\preceq u\preceq\bar{\b}\}=\a'+\langle\bar{\b}\rangle.$$

Let $\bar{\b}\in\F(\s_{\Gamma^p}\a')$ and $\bar{\c}\in\Gamma^p\backslash\a'$ such that $\fpt(\bar{\b})\subseteq\fpt(\bar{\c})$. Hence $\bar{\c}\preceq\bar{\b}$. Since both $\bar{\b}$ and $\bar{\c}$ are facets of $\Gamma^p$, we have $\bar{\b}=\bar{\c}$ and, moreover, $\fpt(\bar{\b})=\fpt(\bar{\c})$.

``If part'': Let $\Gamma^p$ be $k$-decomposable with the shedding face $\a'\in\NN^t$. Define $\a\in\NN^n$ by $\a(i)=\sum^{t_i}_{j=1}\a'(ij)$. In a similar argument to only if part, we can show that $\Gamma$ is $k$-decomposable with the shedding face $\a$.
\end{proof}

Let $\Gamma\subset\NN^n_\infty$ be a multicomplex with facets $\F(\Gamma)=\{\a_1,\ldots,\a_r\}$ where $\a_i\in\{0,\infty\}^n$. For all $i$, set $F_i=\{x_j:\a_i(j)=\infty\}$. We call $\Del=\langle F_1,\ldots,F_r\rangle$ the \emph{simplicial complex associated to} $\Gamma$.

We come to the second main result of the paper which improves Theorem 3.10. of \cite{So}.

\begin{cor}\label{polar of k-clean}
The monomial ideal $I$ is pretty $k$-clean if and only if $I^p$ is $k$-clean.
\end{cor}
\begin{proof}
$I$ is pretty $k$-clean if and only if $\Gamma(I)$ is $k$-decomposable (by Theorem \ref{k-decom pretty}) if and only if $\Gamma(I)^p$ is $k$-decomposable (by Theorem \ref{polar k-decom}) if and only if the simplicial complex $\Del$ associated to $\Gamma(I)^p$ is $k$-decomposable (by Proposition \ref{Sim-multi k-decom}) if and only if $I_\Del=I(\Gamma(I)^p)=I^p$ is $k$-clean (by Theorem \ref{k-dec k-cl}).
\end{proof}

Combining Theorem \ref{pretty k-clean} and Corollary \ref{polar of k-clean} we immediately
obtain the following result which implies that the converse of Theorem 3.3 of \cite{Ra} holds.

\begin{cor}
If $I\subset S$ is a monomial ideal which has no embedded prime ideal. Then $I$ is $k$-clean if and only if $I^p$ is $k$-clean.
\end{cor}

\begin{cor}
Let $I\subset K[X]$ and $J\subset K[Y]$ be two monomial ideals. Then $I$ and $J$ are pretty $k$-clean if and only if $IJ$ is pretty $k$-clean.
\end{cor}
\begin{proof}
Let $I_{\Del_1}=I^p$ and $I_{\Del_2}=J^p$ for some disjoint simplicial complexes $\Del_1$ and $\Del_2$. $I$ and $J$ are pretty $k$-clean if and only if $I^p$ and $J^p$ are $k$-clean (by Corollary \ref{polar of k-clean}) if and only if $\Del_1$ and $\Del_2$ are $k$-decomposable (by Theorem \ref{k-dec k-cl}) if and only if $\Del_1.\Del_2$ is $k$-decomposable (by Theorem \ref{join sim. com}) if and only if $I_{\Del_1.\Del_2}=I^pJ^p=(IJ)^p$ is $k$-clean (by Corollary \ref{join sim. com}) if and only if $IJ$ is pretty $k$-clean (by Theorem \ref{polar of k-clean}).
\end{proof}

\ \\ \\
Rahim Rahmati-Asghar,\\
Department of Mathematics, Faculty of Basic Sciences,\\
University of Maragheh, P. O. Box 55181-83111, Maragheh, Iran.\\
E-mail:  \email{rahmatiasghar.r@gmail.com}


\begin{thebibliography}{1}

%\bibitem{BaSo} S. Bandari, A. Soleyman Jahan, \textit{The cleanness of (symbolic) powers of stanley-eisner ideals}, arXiv:1505.00634v1.

\bibitem{BiPr} L.J. Billera and J.S. Provan, \textit{Decompositions of simplicial complexes related to diameters of convex polyhedra}, Math. Oper. Res. \textbf{5}, no. 4, 576-594 (1980).

\bibitem{Dr} A. Dress, \textit{A new algebraic criterion for shellability}. Beitrage zur Alg. und Geom. \textbf{34}(1) 45-55 (1993).

\bibitem{Ha} M. Hachimori, \textit{Combinatorics of construtible complexes}, Ph.D. thesis, Univ. of Tokyo, Tokyo, (2000).

\bibitem{He} J. Herzog, \textit{A Survey on Stanley Depth}, In: Monomial Ideals, Computations and Applications, Volume 2083 of Lecture Notes in Mathematics, Springer, 2013.

\bibitem{HePo} J. Herzog, D. Popescu, \textit{Finite filtrations of modules and shellable multicomplexes}. Manuscripta Math. \textbf{121} (2006), no. 3, 385–410.

%\bibitem{HePoVl} J. Herzog, D. Popescu, M. Vladoiu, \textit{On the Ext-modules of ideals of Borel type}, Contemporary Math. 331 (2003), 171-186.

%\bibitem{KlMiMiNaPe} J. Kleppe, J. Migliore, R. M. Mir\'{o}-Roig, U. Nagel, and C. Peterson, \textit{Gorenstein liaison, complete intersection liaison invariants and unobstructedness}. Mem. Amer. Math. Soc. \textbf{154} (2001), no. 732.

\bibitem{MoTa} S. Moriyama and F. Takeuchi, \textit{Incremental construction properties in dimension two: shellability, extendable shellability and vertex decomposability}, Discrete Math. \textbf{263}, 295-296 (2003).

%\bibitem{NaRo} U. Nagel, T. R\"{o}mer, \textit{Glicci simplicial complexes},

%\bibitem{Po} D. Popescu, \textit{Criteria for shellable multicomplexes}, An. St. Univ. Ovidius Constanca Vol. 14(2), 2006, 73–84.

\bibitem{Ra} R. Rahmati-Asghar, \textit{$k$-clean monomial ideals}. arXiv:1702.07574v1.

%\bibitem{RaYa} R. Rahmati-Asghar, S. Yassemi, \textit{$k$-decomposable monomial ideals}, Algebra Colloq. \textbf{22} (Spec 1) 745-756 (2015).

%\bibitem{Rau} A. Rauf, \textit{Stanley Decompositions, Pretty Clean Filtrations and Reductions Modulo Regular Elements}. Bull. Math. Soc. Sci. Math. Roumaine \textbf{98} (2007), no 50, 347-354.

\bibitem{Si} R. S. Simon, \textit{Combinatorial properties of ``cleanness''}, J. Algebra, \textbf{167} (1994) 361-388.

\bibitem{So} A. Soleyman Jahan, \textit{Prime filtrations of monomial ideals and polarizations}, J. Algebra, \textbf{312}(2), (2007), 1011-1032.

%\bibitem{SoZh} A. Soleyman Jahan, X. Zheng, \textit{Pretty clean monomial ideals and linear quotients}, J. Comb. Ser. A. \textbf{117}, 104-110 (2010).

\bibitem{St} R.P. Stanley, Combinatorics and Commutative Algebra, Birkh\"{a}user, 1983.

\bibitem{Wo} R. Woodroofe, \textit{Chordal and sequentially Cohen-Macaulay clutters}, Electron. J. Combin. \textbf{18} (2011), no. 1, \#P208.

\end{thebibliography}
\end{document}